\documentclass[12pt]{amsart}
\usepackage{amssymb}
\usepackage{epsfig}
\usepackage{graphicx}

\theoremstyle{plain}
\newtheorem{prop}{Proposition}[section]
\newtheorem{theo}[prop]{Theorem}
\newtheorem{coro}[prop]{Corollary}

\newtheorem{lemm}[prop]{Lemma}

\theoremstyle{definition}
\newtheorem{defi}[prop]{Definition}
\newtheorem{prob}[prop]{Problem}

\newtheorem{rema}[prop]{Remark}

\newtheorem{exam}[prop]{Example}

\newcommand{\Hom}{\mathrm{Hom }}
\newcommand{\Sym}{\mathrm{Sym }}

\newcommand{\NE}{\mathrm{NE}}
\newcommand{\oNE}{\overline{\mathrm{NE}}}

\newcommand{\rN}{\mathrm{N}}
\newcommand{\rH}{\mathrm{H}}

\newcommand{\Aut}{\mathrm{Aut}}

\newcommand{\ra}{\rightarrow}

\newcommand{\bP}{{\mathbb P}}
\newcommand{\bA}{{\mathbb A}}
\newcommand{\bC}{{\mathbb C}}
\newcommand{\bF}{{\mathbb F}}

\newcommand{\bQ}{{\mathbb Q}}
\newcommand{\bR}{{\mathbb R}}

\newcommand{\bZ}{{\mathbb Z}}

\newcommand{\cO}{{\mathcal O}}

\newcommand{\cU}{{\mathcal U}}
\newcommand{\cX}{{\mathcal X}}

\newcommand{\fl}{{\mathfrak{l}}}

\author{Brendan Hassett and Yuri Tschinkel}
\title[Hodge theory and Lagrangians]{Hodge theory and Lagrangian
planes on generalized Kummer fourfolds}

\begin{document}
\date{\today}

\maketitle

\section{Introduction} \label{sect:intro}

Suppose $X$ is a smooth projective complex variety.  
Let $\rN_1(X,\bZ) \subset \rH_2(X,\bZ)$ and $\rN^1(X,\bZ) \subset \rH^2(X,\bZ)$
denote the group of curve classes modulo homological equivalence and the 
N\'eron-Severi group respectively.  The monoids of effective classes
in each group generate cones $\NE_1(X) \subset \rN_1(X,\bR)$ and 
$\NE^1(X) \subset \rN^1(X,\bR)$ with closures $\oNE_1(X)$ and $\oNE^1(X)$,
the {\em pseudoeffective cones}.
These play a central r\^ole in the birational geometry of $X$.

Let $X$ be an irreducible holomorphic symplectic variety, i.e.,
a smooth projective simply-connected manifold admitting a unique
nondegenerate holomorphic two-form.  
Let $\left(,\right)$ denote the Beauville-Bogomolov form on the 
cohomology group $\rH^2(X,\bZ)$, normalized so that it is integral
and primitive.  Duality gives a $\bQ$-valued form on $\rH_2(X,\bZ)$,
also denoted $\left(,\right)$. 
When $X$ is a K3 surface
these coincide with the intersection form.  
In higher-dimensions, the form induces an inclusion
$$\rH^2(X,\bZ) \subset \rH_2(X,\bZ),$$
which allows us to extend $\left(,\right)$ to a $\bQ$-valued quadratic form.

Now suppose that $X$ contains a Lagrangian projective space $\bP^{\dim(X)/2}$;
let $\ell \in \rH_2(X,\bZ)$ denote the class of a line in $\bP^{\dim(X)/2}$, and 
$\lambda=N\ell \in \rH^2(X,\bZ)$
a positive integer multiple;  we can take $N$ to be the index of 
$\rH^2(X,\bZ) \subset \rH_2(X,\bZ)$.  Hodge theory \cite{Ran,Voisin} shows
that the deformations of $X$ containing a deformation of the Lagrangian
space coincide with the deformations of $X$ for which $\lambda \in \rH^2(X,\bZ)$
remains of type $(1,1)$.  Infinitesimal Torelli implies this is a divisor
on the deformation space, i.e., 
$$\lambda^{\perp} \subset \rH^1(X,\Omega^1_X) \simeq \rH^1(X,T_X).$$

Our goal is to establish intersection theoretic properties of $\ell$ for
various deformation-equivalence classes of holomorphic symplectic varieties.
Previous results in this direction include
\begin{enumerate}
\item If $X$ is a K3 surface then $\left(\ell,\ell\right)=-2$.
\item If $X$ is deformation equivalent to the Hilbert scheme of length two subschemes
of a K3 surface then $\left(\ell,\ell\right)=-5/2$ \cite{HTGAFA08}.
\end{enumerate}
Here we prove
\begin{theo} \label{theo:main}
If $X$ is deformation equivalent to the generalized Kummer manifold of dimension
four and $\ell$ the class of a line on a Lagrangian plane in $X$ then
$$\left(\ell,\ell \right)=-3/2.$$
\end{theo}
In Section~\ref{sect:divisibility}, we discuss divisibility properties of $\ell$
and the orbit of $[\ell]$ under the monodromy action.  

This is part of a program described in \cite{HT09} to characterize numerically
classes of extremal rays in holomorphic symplectic manifolds, with a view
towards characterizing ample divisors in terms of the intersection
properties of the Beauville-Bogomolov form on the divisor/curve classes.  
We expect that geometric properties of birational contractions should be
encoded in the self-intersections of their extremal rational curves.
In the examples we have considered,
extremal rays associated to Lagrangian projective spaces have {\em smallest}
self-intersection.  See \cite{HTGAFA99,HTGAFA08} for a detailed discussion
of Hilbert schemes of length-two subschemes of K3 surfaces.  
Markman \cite{Mark3} addresses {\em divisorial} contractions and the associated
reflections.  

The main ingredients in our proof include an analysis of automorphisms of
generalized Kummer varieties, their fixed-point loci, and the resulting
`tautological' Hodge classes in middle cohomology.  (Unlike
the case of length-two Hilbert schemes, the middle cohomology is
not generated by
the second cohomology.)   In particular, these tautological classes arise
from explicit complex surfaces (see Theorem~\ref{theo:hodge}).  
We analyze the saturation of the lattice generated by these tautological
classes in the middle cohomology, and integrality properties
of the quadratic form associated with the cup
product.  Computing the orbit of a Lagrangian plane under the automorphism
group, we obtain a precise characterization of the homology classes that may arise (see
Proposition~\ref{prop:summary} and Theorem~\ref{theo:planes}).

\

{\bf Acknowledgments:} We are grateful to Lothar G\"ottsche, Daniel Huybrechts,
J\'anos Koll\'ar, Marc Nieper-Wisskirchen,
and Justin Sawon for useful conversations and correspondence.
The first author was supported by National Science Foundation Grant 0554491 and 0901645;
the second author was supported by National Science Foundation
Grants 0554280 and 0602333.  We appreciate the hospitality of the American Institute of
Mathematics, where some of this work was done.

\section{Automorphisms of holomorphic symplectic manifolds}

Let $X$ be an irreducible holomorphic symplectic manifold.  Let
$$
\Aut^{\circ}(X) \subset \Aut(X)
$$ 
denote the subgroup of 
holomorphic automorphisms of $X$ acting trivially on $\rH^1(T_X)$
and preserving the symplectic form $\omega$.   
Since $X$ has no vector fields,
this is a discrete group.  
This is equivalent to the automorphisms
acting trivially on $\rH^2(X,\bC)$, as
$$\begin{array}{rcl}
\rH^2(X,\bC)&=&\rH^2(\cO_X) \oplus \rH^1(\Omega^1_X) \oplus \rH^0(\Omega^2_X) \\
            &=& \overline{\bC \omega} \oplus (\omega \otimes \rH^1(T_X)) \oplus \bC\omega.
\end{array}
$$

Let $X'$ be deformation equivalent to $X$, i.e., there
exists a connected complex manifold $B$, with distinguished points
$b$ and $b'$, and a proper family of complex manifolds $\pi:\cX \ra B$
with $\cX_b:=\pi^{-1}(b)=X$ and $\cX_{b'}=\pi^{-1}(b')=X'$.  

\begin{theo} \label{theo:action}
$\Aut^{\circ}(X)$ is a deformation invariant of $X$, i.e., there
exists a local system of groups
$$\Aut^{\circ}(\cX/B) \ra B$$
acting on $\cX \ra B$,
such that for each $b' \in B$ the fiber is isomorphic to
$\Aut^{\circ}(X')$.  
\end{theo}
\begin{proof}
Consider a local universal deformation space of $X$
$$\psi:\cU \ra \Delta,$$
where $\Delta$ is a small polydisk and $\cU_0=X$ \cite{Kur}.  The
completeness of this family implies we can construct
this equivariantly for the action of $\Aut(X)$, i.e.,
$\Aut(X)$ acts on $\cU$ and $\Delta$ and $\psi$ is equivariant with
respect to these actions.  However, $\Aut^{\circ}(X)$ acts
trivially on the tangent space $T_0\Delta=\rH^1(T_X)$,
thus acts trivially on $\Delta$ as well.  It follows that $\Aut^{\circ}(X)$
acts fiberwise on $\cU \ra \Delta$.  

Let $\Aut^{\circ}(\cX/B)$ denote the group over $B$ classifying
automorphisms acting trivially on $\rH^2({\cX_b})$ for each $b\in B$.  
The previous analysis shows $\Aut^{\circ}(\cX/B) \ra B$ is a local homeomorphism.
It remains to show this is universally closed.

We first show that each automorphism specializes to a bimeromorphic mapping
$\phi:X' \dashrightarrow X'$.  The reasoning is identical to the
proof of \cite[4.3]{Hu99},
i.e., that two non-separated points in the moduli space correspond to 
bimeromorphic holomorphic symplectic manifolds.  
Consider a convergent sequence $b_n \ra b'$ and a sequence 
$\alpha_n \in \Aut^{\circ}(\cX_{b_n})$.  The sequence of graphs
$$\Gamma_n:=\Gamma_{\alpha_n} \subset \cX_{b_n} \times \cX_{b_n}$$
admits a convergent subsequence.
Let $\Gamma' \subset X' \times X'$ denote the resulting limit cycle, which
necessarily includes a component $Z$ that maps bimeromorphically to each factor.  
We take $\phi$ to be the bimeromorphic map associated with $Z$.  

The second step is to show that 
$$\phi_*:\rH^2(X',\bZ) \ra \rH^2(X',\bZ)$$
is the identity.  This essentially follows from the description of
the `birational K\"ahler cone' in \cite[\S 4]{Hu03}, but we offer a proof here.
Since $\Gamma'$ is the specialization of a correspondence
acting as the identity on $\rH^2$, we know that $\Gamma'_*=\mathrm{Id}$.  
Express
$$\Gamma'=Z + \sum_i Y_i \subset X' \times X',$$
where the $Y_i$ map to proper analytic subsets of each factor;
let $\pi_1$ and $\pi_2$ denote the projections.  It suffices to
show that $\pi_j(Y_i),j=1,2,$ has codimension $>1$;  then 
$Y_i$ acts trivially on $\rH^2$ and $\phi_*=\Gamma'_*=\mathrm{Id}$ on $\rH^2$.
Since $\phi$ is a bimeromorphic
map from a holomorphic symplectic manifold to itself, it is an isomorphism in 
codimension one and $\phi_*$ is an isomorphism of $\rH^2$.  
Enumerate the $Y_i$ such that $\pi_1(Y_i)$ has codimension one, i.e.,
$Y_1,\ldots,Y_k$;  since $\phi$ is an isomorphism in codimension one, these coincide
with the components such that $\pi_2(Y_i)$ has codimension one 
(cf. the proofs of \cite[2.5,4.2]{Hu03}).  
Furthermore, $\phi$ respects these divisors in the sense
$$\phi_*(\sum_{i=1}^k\pi_1(Y_i))=\sum_{i=1}^k \pi_2(Y_i);$$
indeed, $Y_i$ is ruled over both $\pi_1(Y_i)$ and $\pi_2(Y_i)$.  
In this situation, we have \cite[p. 508]{Hu03}
$$
\sum_{i=1}^k [Y_i]_*(\sum_{i=1}^k \pi_1(Y_i))=-\sum_{i=1}^k b_i \pi_2(Y_i),
\quad b_i>0.
$$
The previous two equations contradict the fact that $\Gamma'$ introduces
the identity on the second homology.

Thus $\phi$ is a bimeromorphic mapping of $K$-trivial varieties respecting K\"ahler cones.
A result of Fujiki \cite{Fu81} implies $\phi$ is an isomorphism.
\end{proof}
\begin{rema}
This is related to unpublished results of Kaledin and Verbitsky \cite[\S 6]{KV},
where such automorphisms of generalized Kummer varieties $K_n(A)$ were used to 
exhibit nontrivial trianalytic subvarieties of $K_n(A)$.  
\end{rema}

\section{Application to generalized Kummer varieties}

Suppose that $X$ is a holomorphic symplectic variety of dimension $2n$,
deformation equivalent to a generalized Kummer variety $K_n(A)$, defined as follows:  Given
an abelian surface $A$ with degree $(n+1)$ Hilbert scheme $A^{[n+1]}$, 
$K_n(A)$ is the fiber over $0$ of the addition map $\alpha:A^{[n+1]} \ra A$.
The Beauville-Bogomolov form is given by
\begin{equation} \label{eqn:BBform}
\rH^2(K_n(A),\bZ) = \rH^2(A,\bZ) \oplus_{\perp} \bZ e, \quad \left(e,e\right)=-2(n+1),
\end{equation}
where $2e$ is the class $E$ of the nonreduced subschemes
\cite[\S 4.3.1]{Yosh01}.  
Each class $\eta \in \rH^2(A,\bZ)$ yields a class in $\rH^2(K_n(A),\bZ)$, i.e.,
the subschemes with some support along $\eta$.
We use $\left(,\right)$ to
embed $\rH^2(X,\bZ) \subset \rH_2(X,\bZ)$ and extend $\left(,\right)$ to a $\bQ$-valued
form on $\rH_2(X,\bZ)$.  
Let $e^{\vee} \in \rH_2(K_n(A),\bZ)$ denote the class of a general ruling of $E$,
i.e., where $n-1$ of the points are fixed and the tangent vector at the $n$th
point varies.  We have
$$e=2(n+1)e^{\vee}, \quad \left(e^{\vee},e^{\vee}\right)=-\frac{1}{2(n+1)}.$$

The following data can be extracted from 
G\"ottsche's generating series for the cohomology of 
Kummer varieties \cite[p. 50]{Gott94}.  The table displays the Betti numbers
$\beta_{\nu}(K_n(A))$ for $n \le 4$:
$$\begin{array}{r|cccc}
n  & 1 & 2 & 3 & 4 \\
\nu \quad   & & & & \\
\hline
0 \quad & 1 & 1 & 1 & 1 \\
1 \quad & 0 & 0 & 0 & 0 \\
2 \quad & 22 & 7 & 7 & 7 \\
3 \quad & 0 & 8 &  8 & 8 \\
4 \quad & 1 & 108 & 51 & 36 \\
5 \quad &  & 8 & 56 & 64 \\
6 \quad &  & 7 & 458 & 168 \\
7 \quad &  & 0 & 56 & 288 \\
8 \quad &  & 1 & 51 & 1046
\end{array}
$$

Consider the group
$$G_n= \bZ/2\bZ \ltimes (\bZ/n\bZ)^4 $$
where $\bZ/2\bZ$ acts on $(\bZ/n\bZ)^4$ via $\pm 1$.  
Let $A$ be an abelian surface and identify $A[n]=(\bZ/n\bZ)^4$.  
Then $A[n]$ acts on $A$ via translation and $\bZ/2\bZ$ acts
on $A$ via $\pm 1$.  Altogether, we get an action 
$$G_n \times A \ra A.$$
Note that the induced action is trivial on $\rH^2(A)$.  Conversely,
any finite group acting on $A$ such that the induced action on $\rH^2(A)$
is trivial is a subgroup of $G_n$ for some $n$.  Indeed, actions fixing $0$
are determined by the associated representation on $\rH^1(A)$, but the only
group acting on $\rH^1(A)$ such that the induced action on $\bigwedge^2(\rH^1(A))=\rH^2(A)$
is trivial is $\pm 1$.  

Theorem~\ref{theo:action} and the description on the second cohomology of generalized
Kummer manifolds yields:
\begin{prop} \label{prop:groupaction}
For each $n>2$, the action of $G_n$ on $A$ induces a natural action 
on $K_{n-1}(A)$, which is trivial
on $\rH^2(K_{n-1}(A))$.  This action extends to a natural action on
any deformation $X$ 
on $K_{n-1}(A)$.
\end{prop}
The action of $G_n$ on $K_{n-1}(A)$ has been considered previously, e.g.,
\cite{Salamon,BNWS}.

\section{Analysis of the cohomology of $K_2(A)$}  
\label{sect:detailed}

We recall the structure of the cohomology ring of $K_2(A)$, following
\cite{Verb96}, \cite[\S 4]{LL}, \cite{NWnote}, and \cite{Salamon}:

\begin{prop} 
\cite[4.6]{LL}  \label{prop:decomp}
Let $X$ be deformation equivalent to $K_2(A)$ for $A$ an
abelian surface.  The Lie algebra $\mathfrak{so}(4,5)$ acts on the 
cohomology $\rH^*(X)$ which admits a decomposition
$$\rH^*(X)= \Sym(\rH^2(X))  \oplus {\bf 1}^{80}_X \oplus (\rH^3(X) \oplus \rH^5(X)),$$
where the middle term is a trivial $\mathfrak{so}(4,5)$-representation
consisting of Hodge cycles in $\rH^4(X)$ and the 
last term carries the structure of a sixteen-dimensional spinor representation.  

Furthermore, the induced decomposition on middle cohomology
\begin{equation} \label{eqn:midcohom}
\rH^4(X)= \Sym^2(\rH^2(X)) \oplus_{\perp} {\bf 1}^{80}_X
\end{equation}
is orthogonal under the intersection form.  
\end{prop}
Decomposition (\ref{eqn:midcohom}) is not the only natural decomposition
of the middle cohomology.  The second Chern class $c_2(X)$ is also an invariant 
element of $\rH^4(X)$ but is not orthogonal to $\Sym^2(\rH^2(X))$ \cite[4.8]{LL};
indeed, we have the Fujiki relation
$$c_2(X)\cdot D_1 \cdot D_2 = a_X \left(D_1,D_2\right), \quad D_1,D_2 \in H^2(X),$$
where $a_X$ is a nonzero constant (which will be computed explicitly below).  Writing 
$$\Sym^2(\rH^2(X))^{\circ}=c_2(X)^{\perp} \cap \Sym^2(\rH^2(X)),$$
we obtain an alternate decomposition
\begin{equation} \label{eqn:midcohom2}
\rH^4(X)= \Sym^2(\rH^2(X))^{\circ} \oplus_{\perp} {\bf 1}^{81}_X.
\end{equation}
By the infinitesimal Torelli theorem, the first summand contains no Hodge classes for general
deformations $X$.  Indeed, since the period map is a local diffeomorphism, $\rH^2(X)$ carries a general
weight-two Hodge structure with the relevant numerical invariants;  such a Hodge
structure is {\em Mumford-Tate generic} \cite[\S 2]{Verb98}.   Hodge cycles in $\Sym^2(\rH^2(X))$
yield invariants for the action of the Hodge group of $\rH^2(X)$, which is the special orthogonal
group for $\rH^2(X)$ under the Beauville-Bogomolov form.  The only such invariants
in $\Sym^2(\rH^2(X))$ are multiples of the dual to the Beauville-Bogomolov form.

Integral classes in the summand ${\bf 1}_X^{81}$ are called
{\em canonical Hodge classes}, because they remain of type $(2,2)$ under
arbitrary deformations.  

\

We construct $81$ distinguished rational surfaces in $X=K_2(A)$, whose classes 
span an $81$-dimensional subspace in $\rH^4(X,\bQ)$ which contains the
summand ${\bf 1}^{80}_X$ above.  This 81-dimensional space is different from the 
subspace indicated in (\ref{eqn:midcohom2}), but we will describe explicitly
how they are related.  

For each $\tau \in A$, let $W_{\tau}$ denote the subschemes in
$A^{[3]}$ supported entirely
at $\tau$, with the induced reduced scheme structure.  Brian\c{c}on \cite[p. 76]{Br}
gives explicit equations for the corresponding subscheme of the Hilbert
scheme, via local coordinates.  Eliminating embedded components from these equations,
we find that $W_{\tau} \simeq \bP(1,1,3)$, realized as a cone over a twisted cubic in $\bP^4$.  

In the case where $\tau \in A[3]$, we have $W_{\tau} \subset K_2(A)$.  This
yields $81$ disjoint copies of $\bP(1,1,3)$.   We recover some of their intersection
properties:
\begin{itemize}
\item{$W_{\tau}^2=3$;}
\item{$W_{\tau}e^2=3$;}
\item{$W_{\tau}c_2(K_2(A))=-1$.}
\end{itemize}
We prove these assertions.  Using the fibration 
$$\begin{array}{rcl}
K_2(A) & \ra & A^{[3]} \\
	&    & \downarrow \Sigma \\
        &    &   A
\end{array}
$$
we can reduce the computation of the first number to the main result of \cite{ES98}.  
For the second number, it suffices to check that
$$\cO_{K_2(A)}(e)|W_{\tau}=\cO_{\bP(1,1,3)}(-H)$$
where $H$ is the hyperplane class associated with $\bP(1,1,3) \hookrightarrow \bP^4$.
Let $R$ denote the extremal ray corresponding to the
generic fiber of the diagonal divisor (in $K_2(A)$) over the diagonal of the
symmetric product.  
The diagonal divisor has class $2e$ and $2e\cdot R=-2$, as the symmetric product
has $A_1$-singularities at the generic point of the diagonal.  Thus $e\cdot R=-1$.  
However, $R$ specializes to some cycle of rational curves $R_0 \subset W_{\tau}$ and
$e|W_{\tau}$ is a nontrivial Cartier divisor, i.e., $e=-nH$ for some $n$.  Thus
$$-1=e\cdot R = -n H \cdot R_0,$$
hence $n=1$ and $R_0$ is a ruling of $\bP(1,1,3)$.  
For the third number, consider the induced morphism
$$\phi:\bF_3 \ra \bP(1,1,3) \simeq W_{\tau} \hookrightarrow K_2(A)$$
from the Hirzebruch surface resolving $W_{\tau}$;  let $\Sigma_0$
denote the $(-3)$-curve contracted by $\phi$.  We have the exact sequence
$$0 \ra T_{\bF_3} \ra \phi^* T_{K_2(A)} \ra N_{\phi} \ra 0$$
which implies
$$c_2(\phi^*T_{K_2(A)})=c_2(N_{\phi})-4.$$
We may interpret 
$$c_2(N_{\phi})+\{\text{contribution of $\Sigma_0$ to the double point formula} \}=W_{\tau}^2=3,$$
and an excess-intersection computation \cite[ch.~9]{Fulton} shows that contribution
of $\Sigma_0$ is zero.  We conclude that $c_2(\phi^*T_{K_2(A)})=-1$.  

\begin{rema}[contributed by M. Nieper-Wisskirchen]
These computations can also be put into the framework of Nakajima's
\cite{Nakajima} description of the cohomology of the Hilbert scheme.

Consider the Galois cover 
$$\begin{array}{rcl}
v: A \times K^{n - 1}(A) &\ra& A^{[n]} \\
 (a, \xi) &\mapsto& a + \xi
\end{array}
$$
where $a + \xi$ denotes the
translate of $\xi$ by $a$. 
The Galois group is $A[n]$. Let $B_0$ be the Brian\c{c}on
variety in $A^{[n]}$
of those subschemes whose support is $0$. Then $v^* B_0$ is the
same as $0 \times W$ in the cohomology ring of $A \times  K^{n - 1}(A)$, 
where $W=\sum_{\tau \in A[n]}W_{\tau}$.
Let $B$ be the Brian\c{c}on variety in $A^{[n]}$ of those subschemes
that are supported at a single arbitrary point.  Then $v^* B = A \times W$. 

Now assume $n=3$.  To show that 
$W_\tau^2 = 3$, we have to show $(v^* B_0) \cdot (v^* B) = 81 \cdot 3$, as the
different $W_\tau$ are orthogonal.  As $\deg(v)=81$, this is the same
as showing that $B_0 \cdot B = 3$ on the Hilbert scheme. 
Now $B_0$ happens to be
$\left.q_3(\omega)|{\mathbf 0}\right>$, where $\omega$ is the codimension-four
class of a point in $A$, and
$B$ is $\left.q_3(1)|{\mathbf 0}\right>$, using the standard notation for the cohomology classes
in the Nakajima basis \cite[\S 4]{BoissiereNW}.  Using the well-known commutation relations between
these operators, one gets that the Poincar\'e duality pairing of
$\left.q_3(\omega)|{\mathbf 0}\right>$ and $\left.q_3(1)|{\mathbf 0}\right>$
on the Hilbert scheme is exactly $3$.  

Intersections with $c_2$ can be obtained using the techniques of \cite{Boissiere,BoissiereNW}.
\end{rema}

Consider the averaged class
$$W:=\sum_{\tau \in A[3]} W_{\tau}$$
with intersection properties
\begin{itemize}
\item{$W^2=81 \cdot 3 = 243$;}
\item{$W e^2 = 243$;}
\item{ $W c_2(K_2(A))=-81$.}
\end{itemize}
It follows that
$$W=\frac{3}{8}(c_2(X)+3e^2).$$

Given $p \in A$ consider the locus
$$Y_p=\{(a_1,a_2,p):a_1+a_2+p=0 \} \subset K_2(A).$$
Exchanging the first two terms induces
$$\iota: a_1 \mapsto -p-a_1$$
on the first factor,
which is conjugate to the sign involution.  (Set $a_1=q+\alpha$
for $2q=-p$;  then $\iota(\alpha)=-\alpha$.)   In addition, $Y_p$ meets 
the boundary at  $\{(-2p)+p+p \}.$
Thus $Y_p$ is isomorphic to the Kummer surface $K_1(A)$ blown up at one point,
i.e., the images of $\alpha=\pm 3q$.

The diagonal divisor meets $Y_p$ along the $16$ distinguished $(-2)$-curves and 
with multiplicity two along the $(-1)$-curve over the center, thus we have
$$e^2Y_p=\frac{1}{4}(-4+16\cdot(-2))=-9.$$
Evidently, $Y_p$ is disjoint from $W_{\tau}$ for $p \not \in A[3]$, thus
$$Y_p \cdot W_{\tau}=0.$$
Using the computations for $W_{\tau}$ done previously, we find
$$Y_p=\frac{1}{72}(3c_2(K_2(A))+e^2).$$
In particular, $Y_p^2=1$ and given $f,g \in \rH^2(A,\bC)$ we have
$$f\cdot g \cdot Y_p=2 f \cdot g.$$

Assume $p \not \in A[3]$ and write
$$Z_{\tau}=Y_p-W_{\tau}.$$
As a consequence, we deduce 
$$Z_{\tau}\cdot D_1 \cdot D_2=2\left(D_1,D_2\right)$$
for all $D_1,D_2 \in \rH^2(K_2(A),\bZ)$.
By the orthogonal decompositions (\ref{eqn:midcohom}) and (\ref{eqn:midcohom2}), we conclude
$Z_{\tau} \in {\bf 1}_{K_2(A)}^{81}$, i.e.,
they are canonical Hodge cycles.   
We summarize this analysis as follows:
\begin{prop} \label{prop:canonHodge}
Let $X$ be deformation equivalent to $K_2(A)$, for $A$ an abelian surface.  
Consider the lattice of canonical Hodge classes
$$({\bf 1}^{81}_X \cap \rH^4(X,\bZ)) \subset \rH^4(X,\bZ),$$
associated with the decomposition (\ref{eqn:midcohom2}).  The
classes $\{Z_{\tau} \}_{\tau \in A[3]}$ in this lattice span
$({\bf 1}^{81}_X \cap \rH^4(X,\bQ))$ and satisfy
\begin{equation} \label{eqn:Ztau}
Z_{\tau}^2=4, \quad Z_{\tau}\cdot Z_{\tau'}=1 \text{ if }\tau\neq \tau', \quad c_2(X)Z_{\tau}=28.
\end{equation}
\end{prop}

This raises a question:  How do we interpret these geometrically?  
One approach is to analyze the limit $\lim_{p\ra \tau} Y_p$, to determine whether
it contains $W_{\tau}$ as an irreducible component, with residual intersection 
in the summand ${\bf 1}_X^{81}$.  

The approach we take is to analyze the action of $G_3$ on $K_2(A)$ and its deformations.
Regard $G_3 \subset \mathrm{Sp}(A[3]) \ltimes A[3]$,
the canonical semidirect product associated to the action of the symplectic group on
$A[3]$.  
\begin{theo}[Hodge conjecture for canonical classes] \label{theo:hodge}
Consider the action of $G_3$ on $K_2(A)$.  
For each $\tau \in A[3] \subset G_3$, let $\iota_{\tau}\in G_3$ denote the 
involution fixing $\tau$.
\begin{enumerate}
\item{The fixed point locus of $\iota_{\tau}$ has two
irreducible components.  First, there is an isolated point corresponding to the 
vertex of $W_{\tau}$.  Second, there is a Kummer surface that is an irreducible
component of $\lim_{p \ra \tau} Y_p$ and has class $Z_{\tau}$.  For instance, we have
$$Z_0=\overline{\{(a_1,a_2,a_3):a_1=0, a_2=-a_3, a_2 \neq 0 \} }$$
and the other $Z_{\tau}$ are translates of $Z_0$ via the action of $A[3]$.}
\item{For each deformation $X$ of $K_2(A)$, the $Z_{\tau}$ deforms to a submanifold
of $X$.}  
\end{enumerate}
\end{theo}
\begin{proof}
Each $\iota_{\tau}$ is conjugate to the involution induced by the sign
involution on $A$;  its fixed points are clearly the stratum $Z_0$
and the vertex of the exceptional divisor $W_0$.  

Proposition~\ref{prop:groupaction}
shows each $\iota_{\tau}$ carries over to deformations of $X$;  thus the
fixed-point loci carry over as well.  
\end{proof}

\begin{defi}
Let $X$ be an irreducible holomorphic symplectic manifold, with its
natural hyperk\"ahler structure.  A submanifold $Z\subset X$ is {\em trianalytic}
if $Z$ is analytic with respect to each of the associated complex structures.
\end{defi}
These have been studied systematically by Verbitsky;  see, for example,
\cite{Verb95,Verb98}.  One general result is that analytic subvarieties
representing canonical Hodge classes are automatically trianalytic \cite[Thm.~4.1]{Verb95}.
In particular, {\em all} analytic subvarieties of (Mumford-Tate) general
irreducible holomorphic symplectic manifolds have this property.  
The deformations of the $Z_{\tau}$ in $X$ are thus examples of trianalytic submanifolds.  
We lack a clear picture of what these generic deformations look like.  
However, their middle cohomology should have a piece that is isogenous to $\rH^2(X)$.  
For more discussion of trianalytic subvarieties of generalized Kummer varieties,
see \cite[\S 6]{KV}.

\section{The lattice of canonical classes}
\begin{prop} \label{prop:bigdisc}
The lattice $\Pi=\left<Z_{\tau}:\tau \in A[3] \right>$ under the intersection
form is positive definite of discriminant
$2^2 \cdot 3^{81} \cdot 7$.
We have 
$$c_2(X)=\frac{1}{3} \sum_{\tau \in A[3]} Z_{\tau},
$$
which is non-divisible.  
\end{prop}
\begin{proof}
From Equation~(\ref{eqn:Ztau}), we see that the intersection form of $\Pi$ is
$$\begin{array}{c|cccc}
   & Z_{\tau_1} & Z_{\tau_2} & Z_{\tau_3} & \cdots \\
\hline
Z_{\tau_1} & 4 & 1 & 1& \cdots \\
Z_{\tau_2} & 1 & 4 & 1 &\cdots \\
Z_{\tau_3} & 1 & 1 & 4 & \ddots \\
\vdots     & \vdots & \vdots & \ddots & \ddots
\end{array}
$$
so the corresponding matrix has eigenvalues $84$ (with multiplicity one)
and $3$ (with multiplicity eighty).  

The non-divisibility follows from 
$$c_2(X)\cdot Z_j=28, \quad c_2(X)^2=756, \quad c_2(X)\cdot \bP^2=-3,$$
where $\bP^2 \subset X$ is a Lagrangian plane.  
\end{proof}

\begin{prop} \label{prop:monobig}
Consider the generalized Kummer varieties deformation equivalent to $K_2(A)$.  The image of the monodromy representation on
$$\Pi=\left<Z_{\tau}:\tau \in A[3] \right> \subset \rH^4(X,\bZ)$$
contains the semidirect product
$$\mathrm{Sp}(A[3]) \ltimes A[3],$$
where $\mathrm{Sp}(A[3])$ is the symplectic group.  These groups act on the $Z_{\tau}, \tau \in A[3]$ via permutation.
\end{prop}
\begin{proof}
The monodromy representation for abelian surfaces acts on their three-torsion via
the symplectic group $\mathrm{Sp}(A[3])$;  this group acts on $\Pi$
as well via the permutation representation on the $Z_{\tau}$.  In addition, translation by 
a three-torsion element $\tau' \in A[3]$ induces a non-trivial action on $K_2(A)$;  the
action on $\Pi$ corresponds to the permutation
$$Z_{\tau} \mapsto Z_{\tau+\tau'}.$$
\end{proof}

\section{Comparison between Hilbert schemes and generalized Kummer varieties}
The middle cohomology of the Hilbert scheme $Y$ of length-two subschemes
of a K3 surface is generated by the second cohomology.  Thus the class
of a Lagrangian plane in $Y$ can be written as a quadratic polynomial
in $\rH^2(Y)$.  After suitable deformation, it can be written
as a linear combination of $c_2(Y)$ and $\lambda^2$, where $\lambda$
is proportional to the class of a line.  This {\em Ansatz}
allowed us to compute $\left(\lambda,\lambda\right)$ in this
case \cite[\S 4]{HTGAFA08}.  

In our situation, the presence of $80$ additional classes complicates
the algebra.   The following proposition shows these must be included
in any formula for the class of a Lagrangian plane:

\begin{prop}
Let $X$ be deformation equivalent to $K_2(A)$ and $P \subset X$ a Lagrangian
plane.  Let $\lambda \in \rH^2(X,\bZ)$ be $6[\ell]$ where 
$\ell \subset P$ is a line.  Then we cannot write
$$[P]=a\lambda^2 + b c_2(X)$$
for any rational $a,b \in \bQ$.  
\end{prop}
\begin{proof}
Recall that \cite[p. 123]{SawonThesis}
$$c_2(X)^2=756$$
and \cite[\S 5]{BritzeNieper},\cite[Thm. 6]{Britze}
$$\chi(\cO_X(L))=3 \binom{\frac{\left(L,L\right)}{2}+2}{2}=L^4/4!+L^2c_2(X)/24+\chi(\cO_X).$$
Thus we find
$$L^4=9\left(L,L\right)^2, L^2c_2(X)=54\left(L,L\right).$$

The condition $[P]^2=3$ translates into
$$3=9a^2 \left(\lambda,\lambda\right)^2+ 108ab\left(\lambda,\lambda\right) + 756 b^2.$$
The condition $c_2(X)|\bP^2=-3$ yields
$$-3=54a\left(\lambda,\lambda\right) + 756 b.$$
Finally, the fact that $\lambda|\bP^2=\left(\lambda,\ell\right)\ell$ implies
$$\left(\lambda,\lambda\right)^2/36=9a\left(\lambda,\lambda\right)^2 + 54b\left(\lambda,\lambda\right).$$
These equations admit no solutions for $a,b,\left(\lambda,\lambda\right) \in \bQ$.  
\end{proof}

\section{A key special case}
\label{sect:key}

In order to formulate a revised {\em Ansatz}, we analyze a specific example:
Let $A=E_1 \times E_2$ and consider the planes $P \subset K_2(A)$ associated with the
linear series $3\cdot (0)$ in $E_1\times 0$.
Let $\Lambda'=E_1[3] \times 0 \subset A[3]$,
$\ell$ the class of a line in $P$, and
$$\lambda=6\ell=6E_1-3e$$
the corresponding class in $\rH^2(K_2(A),\bZ)$.  
\begin{prop} \label{prop:formula}
$$[P]=\frac{1}{216}\lambda^2 + \frac{1}{8} c_2(K_2(A))-\frac{1}{3} \sum_{\tau \in \Lambda'}Z_{\tau}$$
\end{prop}
\begin{proof}
We start with the Ansatz
$$[P]=a\lambda^2 + b c_2(X) + \widehat{Z}$$
where $\widehat{Z} \in c_2(X)^{\perp} \cap \left<Z_{\tau}: \tau \in A[3]\right>.$

Recall that \cite[p. 123]{SawonThesis}
$$c_2(X)^2=756$$
and \cite[\S 5]{BritzeNieper},\cite[Thm. 6]{Britze}
$$\chi(\cO_X(L))=3 \binom{\frac{\left(L,L\right)}{2}+2}{2}=L^4/4!+L^2c_2(X)/24+\chi(\cO_X).$$
Thus we find
$$L^4=9\left(L,L\right)^2, \quad L^2c_2(X)=54\left(L,L\right).$$
Furthermore, we have
$$L_1L_2L_3L_4=3 \left[ \left(L_1,L_2\right)\left(L_3,L_4\right) + \left(L_1,L_3\right)\left(L_2,L_4\right)
+ \left(L_1, L_4\right) \left(L_2,L_3\right) \right].$$

The condition $[P]^2=3$ translates into
$$3=9a^2 \left(\lambda,\lambda\right)^2+ 108ab\left(\lambda,\lambda\right) + 756 b^2+\widehat{Z}\cdot \widehat{Z}.$$
The condition $c_2(X)|\bP^2=-3$ yields
$$-3=54a\left(\lambda,\lambda\right) + 756 b.$$
Finally, the fact that $\lambda|\bP^2=\left(\lambda,\ell\right)\ell$ implies
$$\left(\lambda,\lambda\right)^2/36=9a\left(\lambda,\lambda\right)^2 + 54b\left(\lambda,\lambda\right).$$
Given that $\left(\lambda,\lambda\right)=-54$, the last two equations allow us to solve for
$a$ and $b$, i.e., we find $a=1/216$ and $b=1/72$.  Thus we have
$$[P]=\frac{1}{216}\lambda^2 + \frac{1}{72} c_2(K_2(A))+\widehat{Z}.$$
It follows that
$$\begin{array}{rcl}
[P]\cdot e^2&=&\frac{1}{216}\lambda^2e^2 + \frac{54}{72} \left(e,e\right) \\
&=&\frac{3}{216}(\left(\lambda,\lambda\right) \left( e,e \right) + 2 \left(\lambda,e\right)^2)-9/2\\
&=&9.
\end{array}
$$

We apply the formulas from Section~\ref{sect:detailed}:  We know that 
$$P\cdot [W]=P \cdot \frac{3}{8}(c_2(X)+3e^2)=-9/8+81/8=9$$
and 
$$P\cdot [Y_p]=P\cdot \frac{1}{72}(3c_2(X)+e^2)=0.$$
The geometry of $P$ shows it is disjoint from $W_{\tau}$ for $\tau \not \in \Lambda'$.
By symmetry, we conclude that
$$P\cdot [W_{\tau}]=1, \quad \tau \in \Lambda'.$$
It follows that
$$P\cdot Z_{\tau}=\begin{cases} 0 \text{ if } \tau \not \in \Lambda' \\
				-1 \text{ if } \tau \in \Lambda',
			\end{cases}
$$
which reflects the fact that $P$ fails to intersect these trianalytic varieties transversally.

Thus for $\tau \in \Lambda'$ we have
$$\begin{array}{rcl}
\widehat{Z} \cdot Z_{\tau}&=&\left([P]-\frac{1}{216}\lambda^2-\frac{1}{72}c_2(X)\right)\cdot Z_{\tau} \\
			  &=&-1-\frac{1}{216} 2\left(\lambda,\lambda\right) - \frac{1}{72}28 \\
			  &=& -8/9.
\end{array}
$$

Since the $\tau \in \Lambda'$ appear symmetrically in $\widehat{Z}$, we conclude that
$$\widehat{Z}=\alpha \left( \frac{1}{3}c_2(X)-\sum_{\tau \in \Lambda'} Z_{\tau} \right).$$
We have seen that
$$
Z_{\tau}Z_{\tau'} = \begin{cases} 4 & \text{ if } \tau=\tau' \\
					1 & \text{ if } \tau\neq \tau',
		\end{cases}
$$
thus $\alpha=1/3$ and we conclude
$$[P]=\frac{1}{216}\lambda^2 + \frac{1}{72}c_2(X)+\frac{1}{3}(c_2(X)/3-\sum_{\tau \in \Lambda'} Z_{\tau}),$$
which yields the desired formula.
\end{proof}

\section{Integrality}
\label{sect:integrality}
The example in Section~\ref{sect:key} allows us to obtain further integrality results:
\begin{prop}
Let $\Lambda' \subset A[3]$ be a translate of a non-isotropic two-dimensional subspace.
Then 
$$\frac{1}{8}c_2(X)-\frac{1}{3}\sum_{\tau \in \Lambda'} Z_{\tau}$$
intersects each class of $\Pi$ integrally.  
\end{prop}
This statement follows from the computation in Proposition~\ref{prop:formula} and Proposition~\ref{prop:monobig},
which implies that the monodromy acts transitively on translates of non-isotropic subspaces.
These are precisely the projections of planes described above (and their orbits under
the monodromy action of $\mathrm{Sp}(A[3])$ and translation by $A[3]$) into $\Pi$.  

\begin{prop}
Let $\Pi'$ denote the dual of the lattice $\Pi$, with the $\bQ$-valued intersection
form induced by $\Pi' \subset \Pi \otimes \bQ$.  Then $\Pi'$ is generated by
\begin{itemize}
\item{$\frac{1}{28}c_2(X)=\frac{1}{84}\sum_{\tau} Z_{\tau}$;}
\item{$\frac{1}{3} \sum_{\tau} \beta_{\tau} Z_{\tau},$ where $\beta_{\tau} \in \bZ$ and
satisfy $\sum_{\tau} \beta_{\tau}\equiv 0 \pmod{3}$.}
\end{itemize}
\end{prop}
\begin{proof}
It is evident that these classes generate a subgroup containing $\Pi$ and intersect each of
the $Z_{\tau}$ integrally.  Thus if $M$ is the lattice they generate, we have
$$\Pi \subset M \subset \Pi'.$$
We compute the index of $\Pi$ in $M$.  The sublattice $M_2\subset \Pi'$ generated by the classes 
$\frac{1}{3} \sum_{\tau} \beta_{\tau} Z_{\tau}$ as above
factors
$$\Pi \subset M_2 \subset \frac{1}{3} \Pi.$$
The image of $M_2$ in $\frac{1}{3} \Pi/\Pi\simeq (\bZ/3\bZ)^{81}$ is a hyperplane, thus $\Pi \subset M_2$ has
index $3^{80}$.  Note that $c_2(X) \in M_2$ as a primitive vector, thus
$$
M_2 \subset M_2+\frac{1}{28}c_2(X) = M
$$ 
has index $28$.  Thus the index of $\Pi \subset M$ is $3^{80}\cdot 28$,
which is the discriminant.  
\end{proof}

\begin{prop} \label{prop:Zprime}
Consider the vectors $Z'_{\tau}=\frac{1}{3}(\frac{c_2(X)}{28}-Z_{\tau})$ for each $\tau$, which have
intersections
\begin{equation} \label{eqn:Zetatau}
Z'_{\tau_1}Z'_{\tau_2} = \begin{cases} \frac{83}{28\cdot 9} & \text{ if } \tau_1= \tau_2 \\
					\frac{-1}{28\cdot 9} & \text{ if } \tau_1\neq \tau_2,
		\end{cases}.
\end{equation}
We have
$$\Pi'=\{ \sum_{\tau} \beta_{\tau} Z'_{\tau}: \beta_{\tau} \in \bZ, \sum_{\tau} \beta_{\tau} \equiv 0 \pmod{3} \}.$$
The intersection form on $\Pi'$ takes values in $\frac{1}{28\cdot 9}\bZ$;  squares of elements are in 
$\frac{1}{28 \cdot 3}\bZ$.  
\end{prop}
\begin{proof}
Only the last statement requires verification:  It suffices to observe that the matrix
\begin{equation} \label{eqn:83matrix}
\Psi=\left( \begin{matrix} 83 & -1 & -1 & \cdots \\
		        -1 & 83 & -1 & \ddots \\
		        -1 & -1 & \ddots & \vdots \\
			\vdots & \ddots & \cdots & 83 \end{matrix} \right)
\end{equation}
is congruent to 
$$\left( \begin{matrix} -1 & -1 & -1 & \cdots \\
		        -1 & -1 & -1 & \ddots \\
		        -1 & -1 & \ddots & \vdots \\
			\vdots & \ddots & \cdots & -1 \end{matrix} \right)$$
modulo three.  Elements summing to zero modulo three are in the kernel.
\end{proof}

\section{Computing $\left(\lambda,\lambda\right)$}
We formulate a new {\em Ansatz} in light of the example in Section~\ref{sect:key}.

Let $P$ be a class of a plane, which we now assume can be written
\begin{equation} P=a\lambda^2 + bc_2(X) + \widehat{Z} \label{eqn:Ansatz2}
\end{equation}
as above and in the proof of Proposition~\ref{prop:formula}.  
Note that $\widehat{Z}^2 \ge 0$ as the lattice $\Pi$ is positive definite,
thus 
$$(a\lambda^2+bc_2(X))^2 \le 3.$$

Using the equations 
$$\begin{array}{rcl}
-3&=& 54a\left(\lambda,\lambda\right) + b 756 \\
\frac{\left(\lambda,\lambda\right)^2}{36} &=& 9a\left(\lambda,\lambda\right)^2 + 54b\left(\lambda,\lambda\right) 
\end{array}
$$ 
to eliminate $a$ and $b$, we find
$$
(a\lambda^2+bc_2(X))^2=\frac{1}{46656}(7\left(\lambda,\lambda\right)^2+ 108 \left(\lambda,\lambda\right)+972).$$
This is always positive, and is $\le 3$ only when $|\left(\lambda,\lambda\right)| \le 150$.  
Thus there are 
only finitely many possibilities for $\left( \lambda,\lambda\right)$.

\

We introduce a new class:  Let $\mu \in \bZ \lambda^2 + \bZ c_2(X)$ denote a primitive
class orthogonal to $c_2(X)$ such that $\lambda^2$ appears with a positive coefficient.  
This is unique up to a positive scalar;  we'll choose a specific representative 
in the next lemma.
Note that $\mu$ is also orthogonal to each of the $Z_{\tau}$ and $Z'_{\tau}$,
by Proposition~\ref{prop:decomp}.  
\begin{lemm}
We have
$\mu^2 >0.$
\end{lemm}
\begin{proof}
We may write
$$\mu=7\lambda^2 -\frac{\left(\lambda,\lambda\right)}{2}c_2(X),$$
which is orthogonal to $c_2(X)$.  Using the intersection numbers computed above, we find
$$\mu^2=49 \cdot 9 \left(\lambda,\lambda\right)^2 - 7 \cdot 54 \left(\lambda,\lambda \right)^2
+ 756/4 \left(\lambda,\lambda \right)^2 = 252 \left(\lambda,\lambda \right)^2 .$$
\end{proof}

We write the class of a Lagrangian plane in terms of this orthogonal decomposition;
since $[P]$ is integral, its projection onto each summand lies in the dual
to the underlying lattice.  In particular, we have
\begin{equation} \label{eqn:Ansatz3}
[P]=\alpha \mu + \sum_{\tau} \beta_{\tau} Z'_{\tau}, \beta_{\tau} \in \bZ, \alpha \in \bQ,
\end{equation}
and $\sum_{\tau} \beta_{\tau} \equiv 0 \pmod{3}$.  
If $a$ is the constant in Equation~\eqref{eqn:Ansatz2} we have $a=7\alpha$.  
Combining the equations above with these, we find
$$(\alpha \mu)^2=a^2 \left(\lambda,\lambda \right)^2 \frac{36}{7}.
$$

Write 
$$S=\left(\sum_{\tau} \beta_{\tau} Z'_{\tau}\right)^2=P^2 - (\alpha \mu)^2 
= 3 - a^2 \left(\lambda,\lambda \right)^2 \frac{36}{7}$$
and use the equations above to eliminate $a$.  We find
\begin{align*}
S&=\frac{-49 \left( \lambda,\lambda \right)^2 -756 \left(\lambda,\lambda\right) + 976860 }{326592} \\
&=\frac{-49 \left( \lambda,\lambda \right)^2 -756 \left(\lambda,\lambda\right) + 2^2\cdot 3^6 \cdot 5 \cdot 67 }{2^6 \cdot 3^6 \cdot 7} 
\end{align*}
The lattice $\left< Z'_{\tau} \right>$ is positive definite, so $S>0$.  
This implies 
\begin{equation} \label{eqn:bounds}
-150 \le \left(\lambda,\lambda \right) <134.
\end{equation}

We know that $9\cdot 28 \cdot S$ is an integer represented by the quadratic form (\ref{eqn:83matrix}),
and thus is divisible by $3$.  We therefore obtain the following divisibility
$$2^4 3^5 | 49 \left(\lambda,\lambda\right)^2 + 756 \left(\lambda,\lambda\right) - 2^2\cdot 3^6 \cdot 5 \cdot 67.$$
Thus we find 
\begin{equation} \label{eqn:congruences}
\left( \lambda,\lambda \right) \equiv 2 \pmod{4}, \quad 
\left( \lambda,\lambda \right) \equiv 0 \pmod{27}.
\end{equation}
The only admissible values in the range $(-150,134)$ are
$$\left(\lambda,\lambda\right)=\pm 54.$$
It follows that 
\begin{equation} \label{eqn:Svalue}
S=\left(\sum_{\tau} \beta_{\tau} Z'_{\tau}\right)^2= 75/28
\end{equation}
and $9\cdot 75=675$ is represented by (\ref{eqn:83matrix}).   

\section{Analysis of short vectors}
We complete the proof of Theorem~\ref{theo:main} and extract
additional information on how classes of Lagrangian planes
project onto the tautological classes.  

\begin{lemm}
Let $\sum_{\tau} \beta_{\tau} Z'_{\tau}$ denote the projection of $P$ onto the
sublattice $\Pi'$, as described in (\ref{eqn:Ansatz3}).   Then we have
$\sum_{\tau} \beta_{\tau}=9$.  
\end{lemm}
\begin{proof}
Recall that $c_2(X)[P]=-3$ and $c_2(X)\mu=0$, so that
$$c_2(X) (\sum_{\tau} \beta_{\tau} Z'_{\tau})=-3.$$
On the other hand, we have
$$c_2(X) Z'_{\tau}=c_2(X) (c_2(X)/84-1/3 Z_{\tau})=-1/3$$
so we find
$$c_2(X) (\beta_{\tau} Z'_{\tau})= -1/3 \sum_{\tau} \beta_{\tau}.$$
The equation follows.
\end{proof}

Let $\Psi=9\cdot 28 \cdot \Pi'$, i.e., the dual to $\Pi$ renormalized to
an integral form as in (\ref{eqn:83matrix}).  Then the Lemma implies
$$\begin{array}{rcl}
(\sum_{\tau} \beta_{\tau}Z_{\tau})^2_{\Psi}&=&84(\sum_{\tau} \beta_{\tau}^2)-(\sum_{\tau} \beta_{\tau})^2 \\
					   & &84(\sum_{\tau} \beta_{\tau}^2)-81.
\end{array}
$$
This means that $9\cdot 28 \cdot S-81 \equiv 0 \pmod{84}$, which yields
the congruences
$$-49\left(\lambda,\lambda\right)^2 - 756 \left(\lambda,\lambda\right)+2^2\cdot 3^6\cdot 7 \cdot 51 \equiv 0 \pmod{2^6\cdot 3^5 \cdot 7},$$
which are stronger than those in (\ref{eqn:congruences}).  We obtain that
$$\lambda \equiv 2 \pmod{8}, \quad \lambda \equiv 0 \pmod{27}.$$
Combining this with the bounds (\ref{eqn:bounds}), we conclude that
$$\left(\lambda,\lambda\right)=-54.$$
Since $\lambda=6\ell$, we deduce that $\left(\ell,\ell\right)=-3/2$, which 
completes the proof of Theorem~\ref{theo:main}.  

\begin{prop}
We have
$\beta_{\tau}=0$ for all but nine values of $\tau$, for which $\beta_{\tau}=1$. 
\end{prop}
\begin{proof}
We have shown (see (\ref{eqn:Svalue}) above) that
$S=75/28$, which means that
$$84 \sum_{\tau} \beta_{\tau}^2=81 + 9 \cdot 28 \cdot S=756$$
whence
$$\sum_{\tau} \beta_{\tau}^2=\sum_{\tau} \beta_{\tau}=9.$$
Since the $\beta_{\tau}$ are integers, 
at most $9$ of them are nonzero, and we may
as well restrict to that subspace.  

By basic calculus, the maximal value of the function $x_1+\cdots+x_n$ on the unit sphere
is $\sqrt{n}$, achieved precisely at $x_1=x_2=\cdots=x_n=1/\sqrt{n}$.  Applying this
to our nine-dimensional subspace gives the result.  
\end{proof}

\section{Monodromy and intersection analysis}
\label{sect:monodromy}

We continue to 
identify the set of all $\tau$ with $A[3]$, the three-torsion
of an abelian surface.  The following Proposition
summarizes what we have obtained so far:
\begin{prop} \label{prop:summary}
Let $X$ be deformation equivalent to $K_2(A)$, $P\subset X$ a Lagrangian plane,
$\ell\in \rH_2(X,\bZ)$ the class of a line on the plane, and $\lambda=6\ell \in \rH^2(X,\bZ)$.
Then we have
\begin{eqnarray*}
[P]&=&\frac{1}{216}\lambda^2 + \frac{1}{56} c_2(X)+ \sum_{\tau \in \Lambda} Z'_{\tau}  \\
   &=&\frac{1}{216}\lambda^2 + \frac{1}{8}c_2(X) - \frac{1}{3} \sum_{\tau \in \Lambda} Z_{\tau}, 
\end{eqnarray*}
where $Z'_{\tau}=\frac{1}{3}(\frac{c_2(X)}{28}-Z_{\tau})$ and $\Lambda \subset A[3]$ is
a set of cardinality nine.  
\end{prop}

We recall the following geometric facts:
\begin{itemize}
\item{the group $G_3$ acts on $X$, and thus on 
$\{Z_{\tau}:\tau \in A[3] \}$ and $\Pi$, as described in 
Proposition~\ref{prop:groupaction};}
\item{the $Z_{\tau}$ are canonical trianalytic submanifolds, and they deform as $X$ deforms
(see the end of Section~\ref{sect:detailed});}
\item{the monodromy group acts on $\{Z_{\tau} \}$ and $\Pi$,
as described in Proposition~\ref{prop:monobig};
given a distinguished base point $\tau_0$,  the intersection of the
monodromy group with the stabilizer of $\tau_0$ contains the symplectic group
$\mathrm{Sp}(A[3])$, where $\tau_0$ is interpreted as $0$.}
\end{itemize}

\begin{theo}
\label{theo:planes}
Suppose $P\subset X$ is a Lagrangian plane as in Proposition~\ref{prop:summary}.  Then
$\Lambda \subset A[3]$ is a translate of a two-dimensional subspace.
\end{theo}
Note that the example presented in Section~\ref{sect:key} and the facts on monodromy/automorphisms
quoted above imply that every translate $\Lambda$ of a two-dimensional non-isotropic subspace 
arises from a plane in a deformation of $X$.  

\begin{prob}
Can {\em isotropic} subspaces arise from Lagrangian planes?
\end{prob}

\begin{coro}
Suppose that $P\subset X$ is a plane contained in a manifold deformation equivalent
to $K_2(A)$.  Then the orbit of $P$ under $\Aut(X)$ contains nine distinct planes;  the stabilizer
of $P$ has order dividing nine.  
\end{coro}

\begin{prop}
Let $\Lambda$ denote a set of points arising from a Lagrangian plane.  Then
$$\#\{\Lambda \cap (\Lambda +\tau_0) \} \pmod{3}$$
is constant as $\tau_0$ varies.
\end{prop}
\begin{proof}
We have
$$\left( \sum_{\tau \in \Lambda} Z'_{\tau} \right)^2 \equiv
\left( \sum_{\tau\in \Lambda} Z'_{\tau} \right) 
\left( \sum_{\tau\in(\Lambda+\tau_0)} Z'_{\tau} \right)  \pmod{\bZ},$$
because if $P_2$ is the corresponding plane then
$$P^2 - P\cdot P_2 \in \bZ,$$
which translates into the assertion of the lemma.  

The corresponding vectors have pairing (with respect to $\Psi$)
divisible by $9 \cdot 28$
$$ (\sum_{\tau \in \Lambda} Z'_{\tau}) \cdot_{\Psi} (\sum_{\tau \in (\Lambda + \tau_0)}Z'_{\tau}) \equiv 0 \pmod{9\cdot 28}.$$
Modulo $9\cdot 28$ we find 
$$83 \#\{\tau \in \Lambda \cap (\Lambda+\tau_0)\}-\# \{(\tau,\widetilde{\tau}): \tau\neq \widetilde{\tau},
\tau \in \Lambda, \widetilde{\tau} \in (\Lambda + \tau_0) \}
\equiv 0 $$
and 
$$84 \#\{\tau \in \Lambda \cap (\Lambda + \tau_0)\}-\# \{(\tau,\widetilde{\tau}): 
\tau \in \Lambda, \widetilde{\tau} \in (\Lambda+\tau_0) \}
\equiv 0.$$
Since $|\Lambda|=|\Lambda+\tau_0|=9$,
we have
$$84 \# \{ \tau \in \Lambda \cap (\Lambda+\tau_0) \}
\equiv 81
\pmod{9\cdot 28}$$
and 
$$ \# \{ \tau \in \Lambda \cap (\Lambda+\tau_0) \}
\equiv 0
\pmod{3}.$$
\end{proof}
\begin{rema}
This is insufficient to characterize translates of two-dimensional subspaces
in $A[3]\simeq (\bZ/3\bZ)^4$.  For instance, consider the set
$$\{e_2,e_2+e_1,e_2-e_1,e_3,e_3+e_1, e_3-e_1, e_4, e_4+e_1,e_4-e_1 \}.$$
Every element of the orbit of this set under $G_3$ meets the set in 
$0$ or $3$ points.  
Thus our intersection condition is insufficient to establish 
Theorem~\ref{theo:planes}.
\end{rema}

We strengthen the analysis above.  Suppose $P$ and $\widetilde{P}$ are planes of the form
$$[P]=\alpha \mu + \sum_{\tau} \beta_{\tau} Z'_{\tau}, \quad
[\widetilde{P}]=\alpha \mu + \sum_{\tau} \widetilde{\beta}_{\tau} Z'_{\tau},$$
so the difference is an {\em integral} class
\begin{equation} \label{eqn:difference}
[P]-[\widetilde{P}]=\sum_{\tau} (\beta_{\tau}-\widetilde{\beta}_{\tau})Z'_{\tau}.
\end{equation}
The nonzero coefficients are all $\pm 1$.  
Applying the symplectic group $\mathrm{Sp}(A[3])$, we get additional integral classes.
In particular, the classes
$$\sum_{\tau \in \widetilde{\Lambda}} Z'_{\tau}-Z'_{\tau+\tau_0}$$
are integral for each non-isotropic $\widetilde{\Lambda}\subset A[3]$ and $\tau_0 \in A[3]$.  
\begin{prop}
Let $\widetilde{\Lambda}$ denote a non-isotropic subspace and $\Lambda$ a set of 
points arising from a Lagrangian plane.  Then
$$\# \{\Lambda \cap (\widetilde{\Lambda}+\tau_0)\} \pmod{3}$$
is constant as $\tau_0$ varies.  
\end{prop}
\begin{proof}
Consider the classes $\sum_{\tau \in \Lambda} Z'_{\tau}$, $\sum_{\tau  \in \widetilde{\Lambda}} Z'_{\tau}$,
and $\sum_{\tau \in \Lambda+\tau_0} Z'_{\tau}$.  
We have 
$$\left( \sum_{\tau \in \Lambda} Z'_{\tau} \right)\cdot \left( \sum_{\tau \in \widetilde{\Lambda}+\tau_0} Z'_{\tau}
-\sum_{\tau \in \widetilde{\Lambda}} Z'_{\tau}\right) \in \bZ$$
which translates into 
$$\left( \sum_{\tau \in \Lambda} Z'_{\tau} \right)\cdot_{\Psi} \left( \sum_{\tau \in \widetilde{\Lambda}+\tau_0} Z'_{\tau}
-\sum_{\tau \in \widetilde{\Lambda}} Z'_{\tau}\right) \equiv 0 \pmod{9\cdot 28}.$$
Thus we have
$$\begin{array}{l}
83 \#\{\tau \in \Lambda \cap (\widetilde{\Lambda}+\tau_0)\}-\# \{(\tau,\widetilde{\tau}): \tau\neq \widetilde{\tau},
\tau \in \Lambda, \widetilde{\tau} \in \widetilde{\Lambda}+\tau_0 \}  \\
\equiv
83 \#\{\tau \in \Lambda \cap \widetilde{\Lambda}\}-\# \{(\tau,\widetilde{\tau}): \tau\neq \widetilde{\tau},
\tau \in \Lambda, \widetilde{\tau} \in \widetilde{\Lambda} \} 
\pmod{9\cdot 28}
\end{array}
$$
and
$$\begin{array}{l}
84 \#\{\tau \in \Lambda \cap (\widetilde{\Lambda}+\tau_0)\}-\# \{(\tau,\widetilde{\tau}): 
\tau \in \Lambda, \widetilde{\tau} \in (\widetilde{\Lambda}+\tau_0) \} \\
\equiv 84 \#\{\tau \in \Lambda \cap \widetilde{\Lambda}\}-\# \{(\tau,\widetilde{\tau}): 
\tau \in \Lambda, \widetilde{\tau} \in \widetilde{\Lambda} \}
\pmod{9\cdot 28}.
\end{array}
$$
Since $|\Lambda|=|\widetilde{\Lambda}|=|\widetilde{\Lambda}+\tau_0|=9$, 
we obtain
$$
\#\{\tau \in \Lambda \cap (\widetilde{\Lambda}+\tau_0)\}
\equiv 
\#\{\tau \in \Lambda \cap \widetilde{\Lambda}\}
 \pmod{3},
$$
which is what we sought to prove.
\end{proof}

We shall need the following result on finite geometries, which should
be understood in the context of Radon transforms over finite fields \cite{Zele73}:
\begin{prop}
Let $V$ be a four-dimensional vector space over a finite field
with $q$ elements, with $q$ odd.  Suppose that $V$ admits a 
symplectic form.  Suppose that $\Lambda \subset V$ is a subset with
$q^2$ elements such that, for each affine non-isotropic
plane $\widetilde{\Lambda} \subset V$,
the function
$$\begin{array}{rcl}
V & \ra & \bZ/q\bZ \\
\tau & \mapsto & \# \{ \Lambda \cap (\widetilde{\Lambda} + \tau) \} \pmod{q}
\end{array}
$$
is constant.  Then $\Lambda$ is an affine plane in $V$.
\end{prop}
\begin{proof}
It suffices to show that $\Lambda$ is `convex', in the sense that for any
pair of distinct $\tau_1, \tau_2 \in \Lambda$, the affine line $\fl(\tau_1,\tau_2)$
is contained in $\Lambda$.  Suppose this is not the case, so in particular
$$\#\{ \fl(\tau_1,\tau_2) \cap \Lambda \} < q.$$

For simplicity, assume that $\tau_1=0$ so that every affine plane containing
$\fl(\tau_1,\tau_2)$ is a subspace.  
Let $\mathrm{Gr}(2,V)$ denote the Grassmannian (a smooth quadric hypersurface
in $\bP^5$), $\mathrm{IGr}(2,V)$ the isotropic Grassmannian (a smooth hyperplane
section of $\mathrm{Gr}(2,V)$), and $\Sigma \subset \mathrm{Gr}(2,V)$
the Schubert variety of planes containing $\fl(\tau_1,\tau_2)$ (which is 
isomorphic to $\bP^2$).  Note that $\Sigma \not \subset \mathrm{IGr}(2,V)$,
since the latter is a smooth quadric threefold.  Thus
$$\Sigma_{\circ}:=\Sigma \cap (\mathrm{Gr}(2,V) \setminus \mathrm{IGr}(2,V))
 \simeq \bA^2$$
which has $q^2$ points over our finite field.  

The planes parametrized by $\Sigma_{\circ}$ are disjoint away from $\fl(\tau_1,\tau_2)$.
Hence the pigeon-hole (Dirichlet) principle guarantees there exists at least 
one such plane $\widetilde{\Lambda}$ that contains no points of 
$\Lambda$ outside $\fl(\tau_1,\tau_2)$.  
Otherwise, $\Lambda$ would have more than $q^2$ points.  
Consequently, 
$$m:=\# \{ \Lambda \cap \widetilde{\Lambda} \} = \# \{\Lambda \cap  \fl(\tau_1,\tau_2) \}$$
is between $2$ and $q-1$.  

Consider the affine translates $\widetilde{\Lambda}+\tau$, for $\tau$ taken from
a set of coset representatives of $V/\widetilde{\Lambda}$.  These are $q^2$
disjoint affine planes, each with at least $m$ points of $\Lambda$.  Thus
$\Lambda$ has cardinality at least $mq^2>q^2$, a contradiction.
\end{proof}

\section{Further remarks on saturation}
\begin{prob}
Characterize the saturation of $\Pi$, i.e., the intersection
$$\Pi^{sat}=(\Pi\otimes \bQ) \cap \rH^4(X,\bZ).$$
\end{prob}

\begin{exam}
Let $S=K_1(A)$ be a Kummer surface, $Z_{\tau}$ the $(-2)$-classes
associated with $A[2]$, and $\Pi=\left<Z_{\tau}:\tau \in A[2]\right>$
which has discriminant $2^{16}$.  
The saturation of $\Pi$ in $\rH^2(S,\bZ)$ is computed in \cite[\S 3]{LP};  
an element 
$$\frac{1}{2} \sum_{\tau} \epsilon(\tau) Z_{\tau}, \quad \epsilon(\tau)=0,1 \in \rH^2(S,\bZ)$$
if and only if $\epsilon:A[2] \ra \bZ/2\bZ$ is affine linear.  In particular,
the saturation has discriminant $2^{16-2\cdot 5}=2^6$, since the affine linear functions
have dimension five over $\bZ/2\bZ$.  

We exhibit generators admitting geometric interpretations.  For each $\Lambda \subset A[2]$
non-isotropic of dimension two, the class
\begin{equation} \label{eqn:Kummersurface}
\frac{1}{2}\left( \sum_{\tau \in \Lambda} Z_{\tau}- \sum_{\tau \in (\Lambda+\tau_0)}Z_{\tau}\right)
\end{equation}
is integral for geometric reasons.  Analogously to the example in Section~\ref{sect:key}, we can consider
$A=E_1 \times E_2$ and $\bP^1$'s corresponding to degree-two linear series on the factors.  Taking
difference of one such $\bP^1$ and its translate, we obtain classes of the form (\ref{eqn:Kummersurface}).  

We claim that the extension of $\Pi$ corresponding to these geometric classes agrees with the
extension associated with affine linear forms.  Suppose that $\Lambda$ is non-isotropic through the origin,
defined by $x=y=0$;  over $\bZ/2\bZ$ it can be defined as $x^2+xy+y^2=0$.  Taking the difference of this quadric
and a translate, we find
$$x^2+xy+y^2-(x+a)^2-(x+a)(y+b)-(y+b^2)=-xb-ya-a^2-ab-b^2,$$
which span {\em all} affine linear forms involving $x$ and $y$.  Varying over all such $\Lambda$,
we get the space of all affine linear forms.   
\end{exam}

The same construction is applicable to $K_2(A)$ as well:  The additional classes  (\ref{eqn:difference})
give an extension of $\Pi=\left<Z_{\tau}\right>$ by a subgroup isomorphic to the $\bZ/3\bZ$-vector
space of affine linear forms on $A[3]$.   The resulting lattice has discriminant $2^2 \cdot 3^{71} \cdot 7$.  
However, this is far from the full saturation, because the discriminant of the lattice
$$\Sym^2(\rH^2(X,\bZ)) \cap c_2(X)^{\perp}$$ 
is much smaller.  Indeed, using the formula
$$\begin{array}{rrr}
D_1 D_2 D_3 D_4 &=& 3\left(\left(D_1,D_2\right)\left(D_3,D_4\right) + \left(D_1,D_3\right)\left(D_2,D_4\right)\right. \\
		& &\left.+\left(D_1,D_4\right)\left(D_2,D_3\right)\right)
\end{array}
$$
for $D_1,D_2,D_3,D_4 \in \rH^2(A,\bZ) \subset \rH^2(K_2(A),\bZ)$, we can show that $\Sym^2(\rH^2(K_2(A)))$ has
discriminant $2^{14} \cdot 3^{38}$.  Taking the orthogonal complement to $c_2(X)$ can only increase the
exponent of $3$ by the power of $3$ appearing in $c_2(X)^2=756$.

\section{A negative result on isotropic subspaces}
Theorem~\ref{theo:planes} yields affine subspaces $\Lambda \subset A[3]$ but does specify whether they
are isotropic or non-isotropic (up to translation).  In Section~\ref{sect:key}, we exhibited examples
of non-isotropic subspaces.  Here we identify obstructions to the appearance of isotropic
subspaces:
\begin{prop}   Let $X$ and $\widetilde{X}$ denote manifolds deformation equivalent to $K_2(A)$.
Assume there exist Lagrangian planes $P\subset X$ and $\widetilde{P} \subset \widetilde{X}$ with
\begin{eqnarray*}
[P]&=&\frac{1}{216}\lambda^2 + \frac{1}{56}c_2(X)+ \sum_{\tau \in \Lambda}  Z'_{\tau},\\
\ [\widetilde{P}]&=&\frac{1}{216}{\widetilde{\lambda}}^2 + \frac{1}{56}c_2(X)+ \sum_{\tau \in \widetilde{\Lambda}}  Z'_{\tau},
\end{eqnarray*}
where $\widetilde{\Lambda} \subset A[3]$ is a two-dimensional non-isotropic affine subspace.
Assume $\lambda \in \rH^2(X,\bZ)$ and $\widetilde{\lambda} \in \rH^2(\widetilde{X},\bZ)$ are equivalent
under the monodromy action.  
Then $\Lambda$ is not a translate of an isotropic subspace.
\end{prop}
Recall that 
$\lambda \in \rH^2(X,\bZ)$ satisfies 
$$\left(\lambda,\lambda\right)=-54, \quad
\left(\lambda,\rH^2(X,\bZ)\right)=6\bZ.$$
\begin{proof}
Choose $\gamma$ to be an element of the monodromy group of $X$ 
acting on $A[3]$ via a 
nontrivial symplectic transformation.  Note that 
$$\gamma \Lambda =\Lambda \text{ or } \gamma \Lambda\cap \Lambda = \{0\}$$
which implies
$$(\sum_{\tau \in \Lambda} Z'_{\tau} )\cdot 
(\sum_{\tau \in \gamma \Lambda} Z'_{\tau} ) = \frac{75}{28} \text{ or }
						\frac{1}{84}.$$
Since $P\cdot \gamma(P) \in \bZ$ we find that
$$(\frac{1}{216}\lambda^2 + \frac{1}{56}c_2(X))\cdot 
\gamma(\frac{1}{216}\lambda^2 + \frac{1}{56}c_2(X)) \equiv \frac{9}{28}
					\text{ or } \frac{-1}{84} \pmod{\bZ}.$$

On the other hand, it is possible to produce $\gamma$ where 
$\gamma(\widetilde{\Lambda}) \cap \widetilde{\Lambda}$ is one-dimensional (see Proposition~\ref{prop:monobig}).  
Here, it is crucial that $\widetilde{\Lambda}$ be non-isotropic.  Then
we have
$$(\sum_{\tau \in \widetilde{\Lambda}} Z'_{\tau} )\cdot 
(\sum_{\tau \in \gamma  \widetilde{\Lambda}} Z'_{\tau} ) =  1-\frac{1}{28\cdot 9},$$
which combined with the fact that $\widetilde{P}\cdot \gamma(\widetilde{P}) \in \bZ$
yields 
$$(\frac{1}{216}{\widetilde{\lambda}}^2 + \frac{1}{56}c_2(X))\cdot 
\gamma(\frac{1}{216}{\widetilde{\lambda}}^2 + \frac{1}{56}c_2(X)) \equiv \frac{1}{28\cdot 9}
					\pmod{\bZ}.$$
Since $\lambda$ and $\widetilde{\lambda}$ are in the same orbit under the monodromy representation,
they share common intersection properties.  Thus this is incompatible with the first equation above.  
\end{proof}

\section{Divisibility properties and monodromy}
\label{sect:divisibility}
Let $X$ be deformation equivalent to $K_n(A)$.  
Note that
$$\rH^2(A,\bZ) \simeq U^{\oplus 3}=\left(\begin{matrix} 0 & 1 \\ 1 & 0 \end{matrix} \right)^{\oplus 3},$$
because this is a unimodular lattice of signature $(3,3)$.  
Thus the Beauville-Bogomolov form on $\rH^2(X,\bZ)$ (see (\ref{eqn:BBform})) has 
discriminant group
$d(X)=\Hom(\rH^2(X,\bZ),\bZ)/\rH^2(X,\bZ)\simeq \bZ/2(n+1)\bZ$.  

Let $\Gamma$ denote the subgroup of the orthogonal group of $\rH^2(X,\bZ)$
acting trivially on $d(X)$.  
Consider equivalence classes of primitive vectors, i.e.,
primitive $D,D' \in \rH^2(X,\bZ)$  are equivalent if
$\left(D,D\right)=\left(D',D'\right)$,
$$\left(D,\rH^2(X,\bZ)\right)=\left(D,\rH^2(X,\bZ)\right)=\left<d\right>,$$
and $\frac{1}{d}D=\frac{1}{d}D'$ in $d(X)$.
By \cite[\S 10]{Ei} (cf. \cite[Lemma 3.5]{GHS}) we have that $\Gamma$-orbits are
equal to these equivalence classes.

In light of Markman's results on the monodromy of Hilbert schemes of K3 surfaces \cite{Mark1,Mark2},
it is natural to ask the following:
\begin{prob} 
Does the monodromy representation
on $\rH^2(X,\bZ)$ act transitively on equivalence classes of primitive vectors?
\end{prob}
Assuming this, the class $\ell$ of a line on a Lagrangian plane would have to be in one of the
following equivalence classes:
$$\ell=E-3e^{\vee},3e^{\vee}.$$
We expect only the primitive class occurs because
\begin{itemize}
\item{we know of no examples where the minimal effective generator $\ell$ of an extremal ray
associated to a birational contraction of holomorphic symplectic manifolds fails to
be primitive;}
\item{we conjecture \cite{HT09} that classes of type $e^{\vee}$ arise from {\em divisorial} contractions, e.g.,
the contraction from $K_n(A)$ induced by the Hilbert-Chow morphism 
$$A^{[n+1]} \ra A^{(n+1)}.$$
However, these contractions are preserved under deformations that 
respect $e^{\vee}$ as a Hodge class.}
\end{itemize}

\bibliographystyle{plain}
\bibliography{hodge}

\end{document}